\documentclass[10pt,reqno, english]{amsart}  
\usepackage[utf8]{inputenc}
\usepackage[T1]{fontenc}
\usepackage{amsmath,amsthm}
\usepackage{amsfonts,amssymb,enumerate}
\usepackage{url}
\usepackage{mathtools}  
\usepackage{enumerate, paralist}
\usepackage{anysize}
\usepackage{units}
\usepackage{tikz}

\newtheorem{theorem}{Theorem}[section]
\newtheorem*{theorem*}{Theorem}
\newtheorem*{lemma*}{Lemma}
\newtheorem*{prop*}{Proposition}

\newtheorem{lemma}[theorem]{Lemma}
\newtheorem{cor}[theorem]{Corollary}
\newtheorem{prop}[theorem]{Proposition}
\newtheorem{conj}[theorem]{Conjecture}

\newtheorem{op}[theorem]{Open Problem}

\theoremstyle{definition}

\newtheorem{defn}[theorem]{Definition}

\newcommand{\T}{{\intercal}}
\DeclareMathOperator*{\argmax}{arg\!\max}
\newcommand{\supp}{\text{supp}}
\newcommand{\R}{\mathbb{R}}

\begin{document}

\title[Small Shadows of Lattice Polytopes]{Small Shadows of Lattice Polytopes}



\author[A.~Black]{Alexander E. Black}

\address[AB]{Dept.\ Math., UC Davis, Davis, CA 95616, USA}
\email{aeblack@ucdavis.edu}


\begin{abstract}
 The diameter of the graph of a $d$-dimensional lattice polytope $P \subseteq [0,k]^{n}$ is known to be at most $dk$ due to work by Kleinschmidt and Onn. However, it is an open question whether the monotone diameter, the shortest guaranteed length of a monotone path, of a $d$-dimensional lattice polytope $P = \{\mathbf{x}: A\mathbf{x} \leq \mathbf{b}\} \subseteq [0,k]^{n}$ is bounded by a polynomial in $d$ and $k$. This question is of particular interest in linear optimization, since paths traced by the Simplex method must be monotone. We introduce partial results in this direction including a monotone diameter bound of $3d$ for $k = 2$, a monotone diameter bound of $(d-1)m+1$ for $d$-dimensional $(m+1)$-level polytopes, a pivot rule such that the Simplex method is guaranteed to take at most $dnk||A||_{\infty}$ non-degenerate steps to solve a LP on $P$, and a bound of $dk$ for lengths of paths from certain fixed starting points. Finally, we present a constructive approach to a diameter bound of $(3/2)dk$ and describe how to translate this final bound into an algorithm that solves a linear program by tracing such a path.
\end{abstract}

\maketitle


\section{Introduction}

Despite 75 years of study, the existence of a pivot rule such that the Simplex method runs in polynomial time remains a fundamental open question in the theory of linear programming. A core difficulty in studying this question comes from our lack of understanding of diameters of polytopes. Namely, the Simplex method solves a linear program (LP) by walking along the graph of the polyhedron that forms its feasible region, and the length of the path followed governs the run-time. Furthermore, the path must be \emph{monotone} meaning that each step of the path must improve the linear objective function. The worst-case distance to the unique sink of the directed graph of a polytope across all orientations induced by generic linear objective functions is called the \emph{monotone diameter} of the polytope. Bounding the monotone diameter is a central problem in the study of the run-time in the Simplex method. The focus of our work is on bounding monotone diameters of $(0,k)$-lattice polytopes (i.e., lattice polytopes contained in $[0,k]^{n}$ for some $n$).

 A strengthening of the notion of monotone diameter is the existence of an \emph{edge-pivot rule}, a rule for choosing an improving neighbor among all improving neighbors at a given vertex as discussed in \cite{PivRulesforCircAug}. For non-degenerate LPs, edge-pivot rules and pivot rules for the Simplex method are the same, but they differ in general due to degeneracy. While we do not specify it in the theorem statements, all of our monotone diameter results are constructive in the sense that we not only bound monotone diameters but also find edge-pivot rules for following those paths. In some cases, we are able to do better than an edge-pivot rule and find a proper pivot rule for the Simplex method. An example of this is a shadow pivot rule, which we use to indicate a pivot rule given by a shadow in the language of Section $4$ of \cite{01simplex} and defined here in Section \ref{subsec:shadrule}. For degenerate LPs, we make a distinction between non-degenerate and degenerate pivots. Namely, a non-degenerate pivot is a pivot that corresponds to walking along an edge of the polytope, whereas a degenerate pivot stays at the same vertex. In our work here, we only ever bound non-degenerate steps, since they are the algorithmic analog to monotone diameter bounds.
 
 With all of this in mind, we may state our first result for $k = 2$. In combinatorial optimization, these polytopes often arise as relaxations of $0/1$-polytopes as with the fractional matching polytope for example. In that context, they are called \emph{half-integral polytopes} (see \cite{CombOpt} for many examples). We stick to that language for the following result:

\begin{theorem}\label{mainthm:HalfBound}
Let $P \subseteq \R^{n}$ be a half-integral polytope of dimension $d$. Then the monotone diameter of $P$ is at most $3d$. A bound of $d+2n$ may be attained by an edge-pivot rule.
\end{theorem}

 These bounds are off by a factor of $2$ from the best possible diameter bound of $(3/2)d$ found in \cite{dplattdiam}. We leave whether the gap between these bounds can be improved as an open question. In the case of 0/1-polytopes, we showed in Lemma $3$ of \cite{01simplex} that the bound of $d$ could be improved when the set of vertices all had the same number of nonzero coordinates. For half-integral polytopes, we may prove an analogous result:
 
 \begin{prop}\label{mainprop:Sparsebound}
 Let $P \subseteq \R^{n}$ be a half-integral polytope of dimension $d$, and suppose that the number of coordinates equal to $1/2$ is a constant $n-s$ for each vertex of $P$. Then there is a shadow pivot rule such that the total number of non-degenerate steps taken to solve an LP on $P$ is at most $2s$. 
 \end{prop}
 
 Another class of polytopes analogous to $(0,k)$-lattice polytopes is the set of $(m+1)$-level polytopes, which are polytopes for which every facet direction attains at most $m+1$ different values on vertices. That is, for a polytope $P = \{\mathbf{x}: A\mathbf{x} \leq \mathbf{b}\}$ with vertices $V(P)$ and such that $A\mathbf{x} \leq \mathbf{b}$ is an irredundant description, $P$ is $(m+1)$-level, when $|\{(\mathbf{a}^{i})^{\T} \mathbf{v}: \mathbf{v} \in V(P)\}| \leq m+1$ for all rows $\mathbf{a}^{i}$ of $A$. The facet directions then play a role similar to the standard basis for $(0,m)$-lattice polytopes. This class of polytopes arises in the study of Theta rank and Theta bodies for sum of squares relaxations of LPs and appears often in combinatorial optimization especially when $m = 1$ (see \cite{ThetaRank, TwoLevelEnums, TwoLevelEnums2}). It turns out that $(m+1)$-level polytopes have exactly the desired property to prove monotone diameter bounds.

\begin{theorem}\label{mainthm:klevelb}
The monotone diameter of a $(m+1)$-level polytope is at most $(d-1)m+1$.
\end{theorem}

Note that any polytope is $(m+1)$-level for some sufficiently large $m$. To extend these bounds to lattice polytopes in general, it is desirable to bound the minimal $m$ for which they are $(m+1)$-level. Note that any $(0,k)$-lattice polytope $P = \{\mathbf{x}: A\mathbf{x}\leq \mathbf{b}\} \subseteq \R^{n}$ is $nk ||A||_{\infty}+1$-level, since the number of different values an integral objective function $\mathbf{c}^{\T} \mathbf{x}$ may take on is at most $||\mathbf{c}||_{\infty} n k$ and each facet direction is given by a row of $A$. As a corollary, we have the following bound on monotone diameters. 

\begin{cor} \label{cor:monodiambds}
The monotone diameter of a $d$-dimensional lattice polytope $P = \{\mathbf{x}: A\mathbf{x} \leq \mathbf{b}\} \subseteq \R^{n}$ is at most $(d-1)||A||_{\infty}nk + 1$, when the constraint matrix $A$ is integral. 
\end{cor}

 One parameter used for bounding the diameter of a polytope is the largest absolute value of a subdeterminant of the constraint matrix $A$ denoted by $\Delta$, when the constraint matrix is integral (see \cite{Bonifas, DiscCurv}). In particular, observe that $||A||_{\infty} \leq \Delta$, and $\Delta$ can grow exponentially in $||A||_{\infty}$ in the worst-case. The state of the art bound for monotone diameters in this case is $O(d^{3} \Delta^{2} \ln(d \Delta))$ \cite{DiscCurv}. Allowing for the additional $||A||_{\infty}$ parameter lets us prove a nice bound for a shadow pivot rule for the Simplex method we introduce in Section \ref{subsec:shadrule} called the \emph{lattice shadow pivot rule}. Note that this is a stronger result than from Corollary \ref{cor:monodiambds}. One can actually implement the Simplex method with this pivot rule and use it to solve LPs.

\begin{theorem}\label{thm:actualpivrule}
Let $P = \{\mathbf{x}: A\mathbf{x} \leq \mathbf{b}\} \subseteq \R^{n}$ be a $d$-dimensional $(0,k)$-lattice polytope with $A$ and $\mathbf{b}$ integral. Then the Simplex method with the lattice shadow pivot rule takes at most $dnk||A||_{\infty}$ non-degenerate steps to solve any LP on $P$. In particular, the monotone diameter of a $(0,k)$-lattice polytope is at most $d^{2}k||A||_{\infty}$.
\end{theorem}

 Our bound may be viewed as an improvement upon that found in \cite{DiscCurv} in two senses. First, for $k \ll d$, our bound is stronger as it is fixed parameter quadratic in $d$ and $||A||_{\infty} \leq \Delta$. Secondly, the bound in \cite{DiscCurv} is only in expectation. Thus, there is no guarantee that a path followed by their rule will actually be of the desired length, while our bound is deterministic. Furthermore, the bounds from Corollary \ref{cor:monodiambds} and Theorem \ref{thm:actualpivrule} are essentially the same, but they involve two completely distinct ways of finding a short monotone path. Both strategies may be optimized further to provide two different routes to better bounds. As a final point of interest, all of our bounds and those from \cite{DiscCurv} are found using a shadow pivot rule. We leave open whether the monotone diameters of lattice polytopes may still be bounded by $dk$. However, in pursuit of this question, we found that for certain starting vertices, we may match the $dk$ bound.

\begin{theorem} \label{mainthm:startpt}
Let $P \subseteq \R^{n}$ be a $(0,k)$-lattice polytope. Let $\sigma: [n] \to \pm[n]$ be any signed permutation, and let 
\[\mathbf{x}_{\sigma} = (\text{\normalfont sign}(\sigma(1))\alpha^{|\sigma(1)|}, \text{\normalfont sign}(\sigma(2))\alpha^{|\sigma(2)|}, \dots, \text{\normalfont sign}(\sigma(n))\alpha^{|\sigma(n)|})\]
for $\alpha \geq 2k+1$. Let $\mathbf{v}$ be an $\mathbf{x}_{\sigma}$-maximal vertex. Then the length of the shortest $\mathbf{c}$-monotone path from $\mathbf{v}$ to a $\mathbf{c}$-maximum is at most $dk$ for any choice of $\mathbf{c} \in \R^{n}$.
\end{theorem}

Alongside our results for monotone diameters, we expand upon an approach of Del Pia and Michini in \cite{shortsimppaths} for solving LPs on lattice polytopes. The key observation in their paper is that we may find an algorithm for solving a LP by following short but not necessarily monotone paths on the polytope. Using this perspective, they constructed an algorithm for solving LPs on $(0,k)$-lattice polytopes that traces a path of length $O(d^{4}k\log(dk))$, which runs in strongly polynomial time with the additional parameter $k$. Another interpretation of their result is as a constructive diameter bound, where we have an efficient way to find and follow the path. We improve upon their result with constructive diameter bounds that match the best known diameter bounds in the literature of $O(dk)$ due to Kleinschmidt and Onn in \cite{origlattpolydiam}. We also construct an associated algorithm for linear programming in this way.

\begin{theorem}\label{mainthm:0ksimpmeth}
Let $P \subseteq \R^{n}$ be a $(0,k)$-lattice polytope of dimension $d$, and let $\mathbf{c} \in \R^{n}$. Then one may solve the LP $\max(\mathbf{c}^{\T} \mathbf{x}: \mathbf{x} \in P)$ by solving a sequence of $2n$ LPs on faces of $P$ with the Simplex method and prescribed pivot rules such that the total number of non-degenerate steps taken is at most $d(k+\lfloor k/2\rfloor)$. Furthermore, the path followed across all these LPs is a path in the graph of $P,$ so this algorithm yields a constructive diameter bound of $d(k+ \lfloor k/2\rfloor)$.
\end{theorem}

Note that the path traced in Theorem \ref{mainthm:0ksimpmeth} is a combination of precisely two monotone paths. Namely, the first one maximizes $\mathbf{x}_{\sigma}^{\T}$ for some choice of signed permutation $\sigma$ to reach an $\mathbf{x}_{\sigma}$-maximal vertex in at most $d\lfloor k/2 \rfloor$ steps. From such a vertex, one follows a $\mathbf{c}$-monotone path to solve the original LP. By Theorem \ref{mainthm:startpt}, we may guarantee that this second monotone path is of length at most $dk$ to arrive at a bound of $d(k + \lfloor k/2\rfloor)$.

\subsection{Prior Work}
\label{subsec:prior}

Our primary tool for all of our results is the shadow-vertex pivot rule introduced by Borgwardt in \cite{borgwardt} for the probabilistic analysis of the Simplex method. To start, Borgwardt showed that the expected diameter of a random polytope from a suitable distribution is polynomially bounded with the bounds very recently strengthened in \cite{RandPolyDiam} to $\Theta(n^{2}m^{\frac{1}{n-1}})$, where $n$ is the number of variables and $m$ is the number of inequalities. In other groundbreaking work, the smoothed analysis of algorithms was initiated by Spielman and Teng for studying the run-time of Simplex method in \cite{SpielmanTeng} and also improved upon in \cite{vershynin, smoothedanalysis}. The main result from the smoothed analysis is that the shadow-vertex pivot rule is polynomial in expectation for the Simplex method under perturbation. Namely, one is allowed to perturb the initial data, and the resulting run-time is $O(n^{2}\sqrt{\log(m)}/\sigma^{2})$ with respect to the number of inequalities $m$, number of variables $n$, and the standard deviation of the Gaussian $\sigma \leq \frac{1}{\sqrt{n} \log(m)}$. For a fixed polytope, the shadow-vertex pivot rule for a random shadow still yields bounds that are polynomial assuming additional parameters as mentioned in the introduction such as maximal subdeterminants or more generally discrete curvature parameters discussed in \cite{DiscCurv}. Bounds of this type originated with work of Dyer and Frieze in \cite{dyerfrieze} for diameters of totally unimodular polytopes that were extended to the maximal subdeterminant case and improved upon substantially in \cite{Bonifas}.

A key distinction between our approach to the shadow-vertex rule and all those mentioned thus far is that ours is purely deterministic. There is no mention of probability, and this is key to each of our arguments. The point in each case is to exploit the existence of certain small shadows of these polytopes that we can find explicitly to follow short paths. This perspective also appeared in our previous work for finding optimal pivot rules for the Simplex method for $0/1$-polytopes \cite{01simplex}.

In general, there are gaps in our knowledge between bounds for non-degenerate steps taken by the Simplex method, monotone diameter bounds, and diameter bounds. For example, network flow polytopes satisfy a linear diameter bound as shown in \cite{netflowHirsch}, and they satisfy a quadratic monotone diameter bound using a pivot rule found in \cite{orlin}. However, the existence of a pivot rule or even a monotone diameter bound that matches the linear diameter bound remains open. 

For $(0,k)$-lattice polytopes, the gap is even larger. In that case, Kleinschmidt and Onn found a diameter bound of $dk$ in \cite{origlattpolydiam} that was later improved upon by Del Pia and Michini to $\lfloor (k-\frac{1}{2})d \rfloor$ in \cite{dplattdiam}. Then these bounds were further improved by Deza and Pournin in \cite{implattdiam} to $kd - \lceil\frac{2d}{3} \rceil - (k-3)$. However, the upper bounds do not match the lower bounds outside of a select few choices of $d$ and $k$. Namely, the best known lower bounds are of the form $\Omega(k^{\frac{d}{d+1}})$ for all $k, d \in \mathbb{N}$ from the class of primitive lattice zonotopes found in \cite{LattZonoDiam, PrimZono}. Furthermore, Del Pia and Michini built on their previous work to find an explicit method to build a polynomial length path in \cite{shortsimppaths} finding a bound of at most $O(d^{4}k\log(dk))$. However, their path is not found by any pivot rule for the Simplex method. Paths followed by the Simplex method must be monotone, and their paths need not be monotone. We argue that a monotone diameter bound is still desirable and emphasize the following conjecture:

\begin{conj} \label{conj:monodiam}
The monotone diameter of a $(0,k)$-lattice polytope is bounded by a polynomial in $d$ and $k$.
\end{conj}

Our results show that this conjecture holds in many special cases. The tools we produce here have the potential to be sharpened to prove this conjecture in general and may be applied to bound monotone diameters of other large classes of polytopes.  

\section{Monotone Diameter Bounds}
\label{sec:monodiam}

As in the probabilistic analysis of the Simplex method, our key tool is the shadow-vertex pivot rule. However, we view the shadow-vertex pivot rule from a deterministic perspective. Namely, instead of arguing that a random shadow is expected to be small, we provide an explicit manner to find and follow small shadows. To explain this technique, we rely on the language of \emph{coherent monotone paths} coming from the geometric combinatorics perspective on the shadow-vertex pivot rule that arose from the monotone path polytope construction in \cite{BSFiberPoly}. A coherent monotone path has a corresponding shadow pivot rule as described in Section \ref{subsec:shadrule}.
\begin{defn}
Let $P \subseteq \R^{n}$ be a polytope. Consider a pair of objective functions, $\mathbf{c}^{\T}$ and $\mathbf{d}^{\T}$. Let $F_{0}$ and $F_{1}$ be the $\mathbf{d}$-minimal and $\mathbf{d}$-maximal faces of $P$ respectively. Let $\mathbf{x}^{0}$ and $\mathbf{x}^{\ast}$ be $\mathbf{c}$-maxima of $F_{0}$ and $F_{1}$ respectively. Then a $\mathbf{c}$-coherent $\mathbf{d}$-monotone path on $P$ is a sequences of vertices $[\mathbf{x}^{0}, \mathbf{x}^{1}, \dots, \mathbf{x}^{m} = \mathbf{x}^{\ast}]$ such that
\[\mathbf{x}^{i+1} \in \argmax_{\mathbf{u} \in N_{\mathbf{d}}(\mathbf{x}^{i})} \frac{\mathbf{c}^{\T}(\mathbf{u} - \mathbf{\mathbf{x}^{i}})}{\mathbf{d}^{\T}(\mathbf{u} - \mathbf{\mathbf{x}^{i}})},\]
where $N_{\mathbf{d}}(\mathbf{x}^{i})$ is the set of all $\mathbf{d}$-improving neighbors of $\mathbf{x}^{i}$.
\end{defn}

The following result from \cite{01simplex} for bounding the lengths of coherent monotone paths will be the key tool for proving all of our results in this paper.
\begin{cor}[Corollary 1 in \cite{01simplex}]
\label{cor:lengthbound}
Let $P \subseteq \R^{d}$ be a polytope, and let $\mathbf{c}^{\T}$ and $\mathbf{d}^{\T}$ be objective functions on $P$. Let $\mathbf{v}$ denote the set of vertices of $P$, and let $F$ denote the face that minimizes $\mathbf{d}^{\T}$. Then the $\mathbf{c}$-coherent $\mathbf{d}$-monotone path of length takes at most $|\mathbf{d}^{\T}(V)|-1$ steps to walk from the $\mathbf{c}$-maximum of $F$ to the $\mathbf{c}$-maximum of $P$. 
\end{cor}

Without any additional machinery, we may immediately bound the monotone diameters of half-integral polytopes.

\begin{proof}[Proof of Theorem \ref{mainthm:HalfBound} and Proposition \ref{mainprop:Sparsebound}]
Suppose without loss of generality that $P$ is full dimensional. Furthermore, for simplicity, we may change coordinates so that the polytope has vertices in $\{-1,0,1\}^{d}$. Let $\mathbf{v}$ be a starting vertex, and let $\mathbf{c}^{\T}$ be the objective function. For each subset $S \subseteq [d]$, let $\mathbf{u}_{S} = \sum_{i \in S} \mathbf{u}_{i} \mathbf{e}_{i}$ for $\mathbf{u} \in \R^{d}$. Then we will follow a path via the following algorithm:
\begin{enumerate}
    \item Set $F = P$, $S = \emptyset$, and $i = 0$.
    \item If $\mathbf{v}_{[n] \setminus S} = \mathbf{0}_{[n] \setminus S}$, check for a $\mathbf{c}$-improving neighbor in $F$. If $\mathbf{c}$-improving neighbors exist, set $\mathbf{v}^{i}$ equal to one of the improving neighbors, set $\mathbf{v} = \mathbf{v}^{i}$, and move to step $3$. Otherwise, move to step $4$.
    \item Set $S = \text{supp}(\mathbf{v})$. Set $F = \{\mathbf{x} \in P: \mathbf{x}_{j} = \mathbf{v}_{j} \text{ for all } j \in S\}$. Set $i = i+1$ and move to step $2$.
    \item Set $\mathbf{v}$ to be the $\frac{\mathbf{c}^{\T}}{-(\mathbf{v}^{i})^{\T}}$-maximal neighbor among all neighbors of $\mathbf{v}$ that are both $-\mathbf{v}^{i}$-improving and $\mathbf{c}$-improving. Repeat until no such neighbor exists. 
    \item Return $\mathbf{v}$
\end{enumerate}
We claim that this algorithm follows a monotone path from an initial vertex to a maximal vertex and returns that maximal vertex. By construction, when $\mathbf{v}$ changes, it must change to a $\mathbf{c}$-improving neighbor. Hence, the algorithm must follow a monotone path. Thus, it remains to argue that the path is short and ends at the maximum. 

 Observe that for $\mathbf{v} \in \{-1,0,1\}^{d}$, we have $\mathbf{u}^{\T}\mathbf{v} \leq |\text{supp}(\mathbf{u})|$ with equality if and only if $\mathbf{v}_{i} = \mathbf{u}_{i}$ for all $i \in \text{supp}(\mathbf{u})$. By negating $\mathbf{v}$, we attain an equivalent lower bound, $\mathbf{u}^{\T} \mathbf{v} \geq -|\text{supp}(\mathbf{u})|$, again with equality if and only if $\mathbf{v}_{i} = -\mathbf{u}_{i}$ for all $i \in \text{supp}(u)$. In particular, $\mathbf{u}^{\T}(- \mathbf{u}) = - |\text{supp}(\mathbf{u})|$. Hence, if $\mathbf{u}$ is a vertex of $P$, then $\mathbf{u}$ lies on the $(-\mathbf{u})$-minimal face of $P$, and the vertices of that face $F$ are exactly the set of vertices of $P$ that agree with $\mathbf{u}$ on its support. Namely, $F = \{\mathbf{x} \in P: \mathbf{x}_{j} = \mathbf{v}_{j} \text{ for all } j \in \text{supp}(\mathbf{v})\}$. Since $P$ is full dimensional, $F = P$ precisely when $\mathbf{v} = \mathbf{0}$. 
 
 At the first step of the algorithm, it checks whether $\mathbf{v} = \mathbf{0}$. If $\mathbf{v} = \mathbf{0}$, we move to any improving neighbor $\mathbf{v}^{1}$. Otherwise, we set $\mathbf{v} = \mathbf{v}^{1}$. In either case, we observe that $\mathbf{v}^{1}$ is on the face $F_{1} = \{\mathbf{x} \in P: \mathbf{x}_{j} = \mathbf{v}^{1}_{j} \text{ for all } j \in [d]\}$. Then we check if $\mathbf{v}^{1}$ is maximal on that face or not. If not, we choose $\mathbf{v}^{2}$ on $F_{2}$ that is an improving neighbor of $\mathbf{v}^{1}$. Note that $\mathbf{v}^{1}_{\text{supp}(\mathbf{v}^{1})} = \mathbf{v}^{2}_{\text{supp}(\mathbf{v}^{1})}$ on $F_{1}$, so $\text{supp}(\mathbf{v}^{1}) \subsetneq \text{supp}(\mathbf{v}^{2})$. Then the same process continues on $F_{2}$ until we reach some $F_{k}$ such that $\mathbf{v}^{k}$ is $\mathbf{c}$-maximal on $F_{k}$. Such a $F_{k}$ must be reached within at most $d$ steps, since in the worst case, $\text{supp}(\mathbf{v}) \subsetneq \text{supp}(\mathbf{v}^{1}) \subsetneq \text{supp}(\mathbf{v}^{2}) \subsetneq \dots \subsetneq \text{supp}(\mathbf{v}^{k})$ is a complete flag of subsets.  
 
 At $\mathbf{v}^{k}$, $\mathbf{v}^{k}$ is a $\mathbf{c}$-maximum of $F_{k}$, where $F_{k}$ is the face given by minimizing $-(\mathbf{v}^{k})^{\T}$. Hence, $\mathbf{v}^{k}$ is a starting point of a $\mathbf{c}$-coherent $-\mathbf{v}^{k}$-monotone path $\gamma$. Then $\gamma$ is of length at most the number of different values of $-(\mathbf{v}^{k})^{\T}$ takes on. Since $|\mathbf{u}^{\T} \mathbf{v}^{k}| \leq |\text{supp}(\mathbf{v}^{k})|$ for all $\mathbf{u} \in \{-1,0,1\}^{d}$, $|\{-(\mathbf{v}^{k})^{\T} \mathbf{u}: \mathbf{u} \in V(P)\}| \leq 2|\text{supp}(\mathbf{v}^{k})| + 1$. Thus, by Corollary \ref{cor:lengthbound}, the length of $\gamma$ is at most $2|\text{supp}(\mathbf{v}^{k})|$. Note that step $4$ of the algorithm is to follow $\gamma$. Hence, the algorithm follows a path to reach the optimum of the desired type. Furthermore, that path is of length at most $d +2|\text{supp}(\mathbf{v}^{k})| \leq 3d$. Without the full dimensional assumption, this algorithm still yields an edge pivot-rule with a similar bound but $2|\text{supp}(\mathbf{v}^{k})| \leq 2n$ instead of $2d$ yielding a bound of $d + 2n$.

Suppose now that $|\text{supp}(\mathbf{z})| = s$ for all $\mathbf{z} \in V$. Then no two distinct vertices agree on their support meaning that the algorithm moves to step $4$ without ever taking a step. Thus, the total length of the path is at most $2 |\text{supp}(\mathbf{v})| = 2s$ as desired. Furthermore, this path is a $\mathbf{c}$-coherent $-\mathbf{v}^{k}$-monotone path meaning that it is followed by the corresponding shadow pivot rule.
\end{proof}

Similarly, to bound the monotone diameters of $(m+1)$-level polytopes, we may directly apply Corollary \ref{cor:lengthbound}.

\begin{proof}[Proof of Theorem \ref{mainthm:klevelb}]
Let $P = \{\mathbf{x}: A\mathbf{x} \leq \mathbf{b}\}$ be a $(m+1)$-level polytope, let $\mathbf{v}$ be a vertex of $P$, and let $\mathbf{c}$ be some objective function we hope to optimize. Let $\{\mathbf{a}_{1}, \mathbf{a}_{2}, \dots, \mathbf{a}_{d}\}$ be a linearly independent set of facet defining directions tight at $\mathbf{v}$. Then $F_{i} = \{\mathbf{x} \in P: \mathbf{a}_{j}^{\T}\mathbf{x} = \mathbf{b}_{j} \text{ for all } j \leq i\}$ is a face of $P$ for all $i \in [d]$. Furthermore, $F_{0} = P$, $F_{d} = \mathbf{v}$, and $F_{i}$ is the $\mathbf{a}_{i}$-maximal face of $F_{i-1}$ for each $i \in [d]$.

Since $\mathbf{v}$ is the $\mathbf{c}$-maximum of $F_{d}$, there is a $\mathbf{c}$-coherent $-\mathbf{a}_{d}$-monotone path on $F_{d}$ that starts at $\mathbf{v}$ and moves to a $\mathbf{c}$-maximum of $F_{d-1}$. This path takes at most $1$ step, since $F_{d-1}$ is $1$-dimensional. Continuing this for each $i$, we find a $\mathbf{c}$-coherent $-\mathbf{a}_{i}$-monotone path from a $\mathbf{c}$-maximum of $F_{i}$ to a $\mathbf{c}$-maximum of $F_{i-1}$. Since $F_{0} = P$, concatenating these $\mathbf{c}$-monotone paths yields a $\mathbf{c}$-monotone path from $\mathbf{v}$ to the $\mathbf{c}$-maximum of $P$. Since $P$ is a $(m+1)$-level polytope, $|\mathbf{a}_{i}^{\T} V| \leq m+1$, so each of these paths is of length at most $m$ by Corollary \ref{cor:lengthbound}. Hence, the total length of the path is at most $(d-1)m + 1$. 
\end{proof}

A natural question following this bound is whether it can be improved any further. We leave this as an open question:

\begin{op}
What is the worst-case monotone diameter of a $d$-dimensional $(m+1)$-level polytope in terms of $m$ and $d$?
\end{op}

An answer to this question may allows us to improve upon our monotone diameter bounds for lattice polytopes. As expressed in Corollary \ref{cor:lengthbound}, this bound matches the bound for the lattice shadow pivot rule from Theorem \ref{thm:actualpivrule}. However, unlike that bound, we may tighten this bound further for sparse constraint matrices. Namely, the number of different values $\mathbf{c}$ takes on is at most $||\mathbf{c} ||_{\infty}|\supp(\mathbf{c})| k$. Thus, let $s$ be the maximum size of the support of a row of the constraint matrix $A$. Then, by the same reasoning as before, the polytope is at most $s k ||A||_{\infty} + 1$-level. We then arrive at the following tighter bound.

\begin{cor}
Let $P = \{\mathbf{x}: A\mathbf{x} \leq \mathbf{b}\} \subseteq \R^{n}$ be a $d$-dimensional $(0,k)$-lattice polytope, and let $s$ be the maximum number of nonzero coordinates in any row of the constraint matrix. Then the monotone diameter of $P$ is at most $(d-1)||A||_{\infty}sk+1$.  
\end{cor}

In particular, after fixing all other parameters, this bound becomes linear in $d$ while remaining independent of the number of facets. One context in which this bound applies is for the LP duals of network flow polytopes for integral objective functions. Namely, for a network flow polytope, all columns of the constraint matrix have at most two nonzero entries, and each nonzero entry must have absolute value at most $1$. For the LP dual, the constraint matrix remains totally unimodular, so the resulting polytope is still integral. However, the constraint matrix becomes $A^{\T}$ meaning that each row has at most two nonzero entries. Hence, the bound in this case becomes $(d-1)2k+1$. Thus, we asymptotically match the $O(dk)$ bound for diameters in the special case of LP duals of network flow polytopes. 

\subsection{Shadow Pivot Rule Bound} 
\label{subsec:shadrule}

When Borgwardt defined the shadow-vertex pivot rule in \cite{borgwardt}, the method of choosing the shadow to follow was purposefully made random, since such a choice allowed for his analysis. However, in \cite{01simplex}, we observed that having a deterministic choice of a shadow may be beneficial under certain conditions. In order to understand this concept, we first define a shadow edge-pivot rule.

\begin{defn}
A \emph{shadow edge-pivot rule} for the Simplex method is defined as follows. Start with a LP $\max(\mathbf{c}^{\T}x : x \in P)$. Then a shadow edge-pivot rule first comes with a method to choose an auxiliary vector $\mathbf{d} \in \mathbb{R}^{n}$ in terms of an initial vertex $\mathbf{x}^{0}$ and objective function $\mathbf{c}$ such that $\mathbf{x}^{0}$ is the $\mathbf{c}$-maximum of the $\mathbf{d}$-minimal face of $P$. Then the corresponding shadow pivot rule is given by 
\[\mathbf{x}^{i+1} = \argmax_{\mathbf{u} \in N_{\mathbf{d}}(\mathbf{x}^{i})} \frac{\mathbf{c}^{\T}(\mathbf{u} - \mathbf{x}^{i})}{\mathbf{d}^{\T}(\mathbf{u}- \mathbf{x}^{i})},\]
where $N_{\mathbf{d}}(\mathbf{x}^{i})$ is the set of $\mathbf{d}$-improving neighbors of $\mathbf{x}^{i}$.
\end{defn}

Geometrically, this pivot rule corresponds exactly to following the $\mathbf{c}$-coherent $\mathbf{d}$-monotone path on the polytope. In Section $3.4$ of \cite{01simplex}, we showed that given a shadow edge-pivot rule, there is a corresponding shadow pivot rule for the Simplex method such that the non-degenerate steps taken by the shadow pivot rule are exactly the steps followed by the shadow edge-pivot rule. Thus, to prove Theorem \ref{thm:actualpivrule}, it suffices to exhibit a shadow edge-pivot rule that is guaranteed to always take few steps on a lattice polytope. We accomplish this with the following rule:

\begin{defn}
The \emph{lattice shadow edge-pivot rule} for the Simplex method is defined as follows. Start with any LP $\max(\mathbf{c}^{\T}\mathbf{x} : A\mathbf{x} \leq \mathbf{b})$ and an initial vertex $\mathbf{x}^{0}$. Let $A_{B}$ be any maximal linearly independent set of facet directions tight at $\mathbf{x}^{0}$. Define $\mathbf{d} = -\sum_{\mathbf{a} \in A_{B}} \mathbf{a}$. Then the lattice shadow edge-pivot rule chooses an improving neighbor by
\[\mathbf{x}^{i+1} = \argmax_{\mathbf{u} \in N_{\mathbf{d}}(\mathbf{x}^{i})} \frac{\mathbf{c}^{\T}(\mathbf{u} - \mathbf{x}^{i})}{\mathbf{d}^{\T}(\mathbf{u}- \mathbf{x}^{i})},\]
where $N_{\mathbf{d}}(\mathbf{x}^{i})$ is the set of $\mathbf{d}$-improving neighbors of $\mathbf{x}^{i}$.
\end{defn}

This rule is a shadow edge-pivot rule, since $\sum_{\mathbf{a} \in A_{B}} \mathbf{a}$ is in the interior of the normal cone at $\mathbf{x}^{0}$ and therefore maximized uniquely at $\mathbf{x}^{0}$. Thus, all that remains is to show that the paths followed satisfy the bound. 

\begin{proof}[Proof of Theorem \ref{thm:actualpivrule}]
Let $\mathbf{d}$ be the auxiliary vector chosen by the lattice shadow edge-pivot rule. Then $||\mathbf{d}||_{\infty} \leq \sum_{\mathbf{a} \in A_{B}} ||\mathbf{a}||_{\infty} \leq d ||A||_{\infty}$. Note that the number of different values that $\mathbf{d}^{\T}\mathbf{x}$ takes on on $P$ is at most $||\mathbf{d}||_{\infty} nk +1 \leq dnk||A||_{\infty} + 1$. Then, by Corollary \ref{cor:lengthbound}, the length of the shadow path chosen by $\mathbf{d}$ is at most $dnk||A||_{\infty}$.
\end{proof}

It is possible that these bounds may be improved further for a more careful choice of auxiliary objective function $\mathbf{d}$ or a more involved method of bounding the lengths. We leave both as open questions. 

\section{Constructive Diameter Bounds} 
\label{sec:constdiam}
Our goal in what remains is to extend the ordered shadow pivot rule argument for $0/1$-polytopes in \cite{01simplex} to $(0,k)$-lattice polytopes. To do this, we first need to refine our explanation of $\mathbf{x}_{\sigma}$-maximal vertices. Recall that a signed permutation $\sigma: [n] \to \pm [n]$ is any function that satisfies $\sigma_{+}(i) = |\sigma(i)|$ is a permutation of $[n]$.  Furthermore, in what follows, let $\mathbf{x}_{-i} = -\mathbf{x}_{i}$ to simplify notation. 

\begin{lemma}\label{lem:lexord}
Let $\sigma: [n] \to \pm[n]$ be a signed permutation, and define 
\[\mathbf{x}_{\sigma} = (\text{\normalfont sign}(\sigma(1))\alpha^{|\sigma(1)|}, \text{\normalfont sign}(\sigma(2))\alpha^{|\sigma(2)|}, \dots, \text{\normalfont sign}(\sigma(n))\alpha^{|\sigma(n)|})\]
for some $\alpha \geq 2k+1$. Then $\mathbf{x}_{\sigma}^{\T}$ induces a lexicographic order on the set of lattice points in $[-k,k]^{d}$. That is
\[\mathbf{x}_{\sigma}^{\T}(\mathbf{x}_{1}, \mathbf{x}_{2}, \dots, \mathbf{x}_{d}) < \mathbf{x}_{\sigma}^{\T} (\mathbf{y}_{1}, \mathbf{y}_{2}, \dots, \mathbf{y}_{d})\]
if and only if for some $j \in [n]$, we have $\mathbf{y}_{\sigma^{-1}(j)} -\mathbf{x}_{\sigma^{-1}(j)} > 0$ and $\mathbf{x}_{\sigma^{-1}(i)} = \mathbf{y}_{\sigma^{-1}(i)}$ for all $i > j$. 
\end{lemma}

\begin{proof}
Assume first that $\sigma(i) = i$ for all $i \in [n]$. Let $\mathbf{x}, \mathbf{y} \in [0,k]^{d}$. Then $\mathbf{x}_{\sigma}^{\T} \mathbf{x} \leq \mathbf{x}_{\sigma}^{\T} \mathbf{y}$ if and only if
\[\sum_{i=1}^{n} \mathbf{x}_{i} \alpha^{i} = \mathbf{x}_{\sigma}^{\T} \mathbf{x} \leq \mathbf{x}_{\sigma}^{\T} \mathbf{y} = \sum_{i=1}^{n} \mathbf{y}_{i} \alpha^{i},\]
which is true if and only if 
\[\sum_{i=1}^{n} (\mathbf{y}_{i} - \mathbf{x}_{i}) \alpha^{i} \geq 0.\]
Let $j = \argmax_{i \in [n]} \mathbf{y}_{i} - \mathbf{x}_{i} \neq 0$. Suppose that $\mathbf{y}_{j} - \mathbf{x}_{j} > 0$. Then 
\[(\mathbf{y}_{j} -\mathbf{x}_{j})\alpha^{j} \geq \alpha^{j} \geq \alpha^{j} - \alpha \geq 2k\left(\frac{\alpha^{j} - \alpha}{\alpha -1}\right) = 2k \sum_{i=1}^{j-1} \alpha^{i} \geq \sum_{i=1}^{j-1} (\mathbf{x}_{i} -\mathbf{y}_{i}) \alpha^{i}, \]
since $\mathbf{x}_{i} - \mathbf{y}_{i} \leq k - (-k) = 2k$. Hence, when $\mathbf{y}_{j} - \mathbf{x}_{j} > 0$,
\[\sum_{i=1}^{n} (\mathbf{y}_{i} - \mathbf{x}_{i}) \alpha^{i} = \sum_{i=1}^{j} (\mathbf{y}_{i}-\mathbf{x}_{i}) \alpha^{i} = (\mathbf{y}_{j}-\mathbf{x}_{j})\alpha^{j} - \sum_{i=1}^{j-1} (\mathbf{x}_{i} - \mathbf{y}_{i}) \alpha^{i} > 0.\]
Thus, $\mathbf{x}_{\sigma}$ induces exactly the inverse lexicographic ordering on lattice points $[0,k]^{n}$. Namely, $\mathbf{x}_{\sigma}^{\T}(\mathbf{x}_{1}, \mathbf{x}_{2}, \dots, \mathbf{x}_{n}) < \mathbf{x}_{\sigma}^{\T}(\mathbf{y}_{1}, \mathbf{y}_{2}, \dots, \mathbf{y}_{n})$ if and only if the largest coordinate $j$ on which they differ satisfies $\mathbf{y}_{j} - \mathbf{x}_{j} > 0$. 

Let $\sigma$ be arbitrary. Then we have
\[\mathbf{x}_{\sigma}^{\T} \mathbf{x} = \sum_{i=1}^{n} \text{\normalfont sign}(\sigma(i)) \mathbf{x}_{i} \alpha^{\sigma_{+}(i)} = \sum_{i=1}^{n} \mathbf{x}_{\sigma^{-1}(i)} \alpha^{i}. \]
Thus, $\mathbf{x}_{\sigma}$ induces the same ordering but on $\mathbf{x}_{\sigma^{-1}(i)}$. From which, the more general conditions follow. 
\end{proof}

From the total ordering, we may provide a useful characterization of the $\mathbf{x}_{\sigma}$-maximal vertices of a $(0,k)$-lattice polytope $P$. 

\begin{cor} \label{cor:flag}
Let $P$ be a $(0,k)$-lattice polytope, and let $\sigma$ be a signed permutation. Define a flag of faces $G_{n} \supseteq G_{n-1} \supseteq \dots \supseteq G_{1} \supseteq G_{0}$ by $G_{n} = P$ and $G_{i-1}$ is the $\mathbf{e}_{\sigma^{-1}(i)}$-maximal face of $G_{i}$. Then $G_{0}$ is the unique $\mathbf{x}_{\sigma}$-maximal vertex. 
\end{cor}

\begin{proof}
Observe first that $\mathbf{x}_{\sigma^{-1}(n)} = \mathbf{y}_{\sigma^{-1}(n)}$ for all $\mathbf{x}, \mathbf{y} \in G_{n-1}$, since $G_{n-1}$ is the $\mathbf{e}_{\sigma^{-1}(n)}$-maximal face of $G_{n}$. Then by induction, $\mathbf{x}_{\sigma^{-1}(j)} = \mathbf{y}_{\sigma^{-1}(j)}$ for all $j \geq i$ and for all $\mathbf{x}, \mathbf{y} \in G_{i-1}$. Hence, all coordinates of $G_{0}$ are determined meaning that $G_{0}$ is a vertex. Let $\mathbf{v}$ be an $\mathbf{x}_{\sigma}$-maximal vertex, which is unique, since the lexicographic order from Lemma \ref{lem:lexord} is a total order on $[-k,k]^{d}$. Note that $\mathbf{v}_{\sigma^{-1}(n)} \leq \mathbf{x}_{\sigma^{-1}(n)}$ for all $\mathbf{x} \in G_{n-1}$, since $G_{n-1}$ is the $\mathbf{e}_{\sigma^{-1}(n)}$-maximal face of $P$. At the same time, $\mathbf{v}$ is $\mathbf{x}_{\sigma}$-maximal, so $\mathbf{v}_{\sigma^{-1}(n)} \geq \mathbf{x}_{\sigma^{-1}(n)}$ by the lexicographic order from Lemma \ref{lem:lexord} meaning that $\mathbf{v}_{\sigma^{-1}(n)} = \mathbf{x}_{\sigma^{-1}(n)}$. Suppose that $\mathbf{v} \in G_{j}$. Then $\mathbf{v}_{\sigma^{-1}(i)} = \mathbf{x}_{\sigma^{-1}(i)}$ and $\mathbf{v}_{\sigma^{-1}(j)} \leq \mathbf{x}_{\sigma^{-1}(j)}$ for all $i > j$ and all $\mathbf{x} \in G_{j-1}$. Hence, again by the maximality of $\mathbf{v}$, $\mathbf{v}_{\sigma^{-1}(j)} = \mathbf{x}_{\sigma^{-1}(j)}$, so $\mathbf{v} \in G_{i-1}$. Note also that $\mathbf{v} \in P = G_{n}$, so by induction, $\mathbf{v} \in G_{i}$ for all $i \in [n] \cup [0]$ meaning that $\mathbf{v} = G_{0}$.  
\end{proof}

From this description of the orientation, we may describe exactly how to find a path to an $\mathbf{x}_{\sigma}$-maximal vertex.

\begin{lemma} \label{lem:GIpath}
Let $P$ be a $d$-dimensional $(0,k)$-lattice polytope, and let $\sigma \in B_{n}$ be a signed permutation. Let $\mathbf{u}$ be any starting a vertex. Then the $\mathbf{x}_{\sigma}$-maximal vertex $\mathbf{v}$ may be reached from $\mathbf{u}$ in at most $dk$ steps by applying the greatest improvement pivot rule. Furthermore, for any starting vertex, there is a choice of $\sigma$ such that the greatest improvement path to the $\mathbf{x}_{\sigma}$-maximal vertex is of length at most $d \lfloor k/2 \rfloor$.
\end{lemma}

\begin{proof}
Let $G_{n} \supseteq G_{n-1} \supseteq \dots \supseteq G_{0}$ be the flag of faces from Corollary \ref{cor:flag}. Let $V_{i}$ be the set of vertices of $G_{i}$. Note that $\mathbf{u} \in V_{n}$, since $G_{n} = P$. Let $\mathbf{x} \in G_{n-1}$. Then $\mathbf{x}_{\sigma^{-1}(n)}$ is maximal, so $\mathbf{x}_{\sigma}^{\T}\mathbf{x} > \mathbf{x}_{\sigma}^{\T} \mathbf{y}$ for all $\mathbf{y} \in V_{n} \setminus V_{n-1}$ via the lexicographic ordering in Lemma \ref{lem:lexord}. Then by the same reasoning $\mathbf{x}_{\sigma}^{\T} \mathbf{x} > \mathbf{x}_{\sigma}^{\T} \mathbf{y}$ for all $\mathbf{x} \in V_{i}$ and $\mathbf{y} \in V_{i+1} \setminus V_{i}$. By induction, for all $\mathbf{x} \in V_{n} \setminus V_{i}$ and $\mathbf{y} \in G_{i}$, we have that $\mathbf{x}_{\sigma}^{\T} \mathbf{x} > \mathbf{x}_{\sigma}^{\T} \mathbf{y}$. 

Suppose that $\mathbf{v} \in V_{i} \setminus V_{i-1}$. We just showed that $\mathbf{x}_{\sigma}^{\T} \mathbf{v} > \mathbf{x}_{\sigma}^{\T} \mathbf{y}$ for all $u \in V_{n} \setminus V_{i}$, so all $\mathbf{x}_{\sigma}$-improving neighbors of $\mathbf{v}$ must be in $V_{i}$. Equivalently, all $\mathbf{x}_{\sigma}$-improving neighbors must not be $\mathbf{e}_{\sigma^{-1}(j)}$-improving for $j > i$. Note that, by the lexicographic order from Lemma \ref{lem:lexord}, any $\mathbf{e}_{\sigma^{-1}(i)}$-improving neighbor in $V_{i}$ will dominate one that is not $\mathbf{e}_{\sigma^{-1}(i)}$-improving. Hence, the greatest improving path is $\mathbf{e}_{\sigma^{-1}(i)}$-monotone until reaching $V_{i-1}$. Thus, for $\mathbf{v} \in V_{i}$, the greatest improving path is $\mathbf{e}_{\sigma^{-1}(i)}$-monotone until it reaches $G_{i-1}$. Note that any $\mathbf{e}_{\sigma^{-1}(i)}$-monotone path is of length at most $k$, since $P$ is a $(0,k)$-lattice polytope. Since $\{G_{i}\}$ is a flag of faces, at most $d$ of $\{G_{i}: i \geq 1\}$ are distinct. Hence, the total length of the path is at most $dk$.  

To find a nearby $\mathbf{x}_{\sigma}$-maximal vertex for some choice of $\sigma$, we may optimize this strategy. First choose the coordinate $j_{1}$ such that $\mathbf{v}_{j_{1}}$ is closest to $0$ or $k$. If it is closest to $k$, follow an $\mathbf{e}_{j_{1}}$-monotone path to $\mathbf{e}_{j_{1}}$-maximal face $G_{n-1}$. If it is closest to $0$, follow any $-\mathbf{e}_{j_{1}}$-monotone path to the $-\mathbf{e}_{j_{1}}$-maximal face $G_{n-1}$. Continue this for each $G_{i}$, while choosing the coordinate closest to $0$ or $k$ that has not already been chosen. This will yield a flag of faces $G_{n} \supseteq G_{n-1} \supseteq \dots \supseteq G_{0}$ of the desired type. At each step there must be a coordinate that varies within $\lfloor k/2 \rfloor$ of $0$ or $k$. Hence, the total length of the path is at most $d \lfloor k/2\rfloor.$ 
\end{proof}

We may furthermore use our description of $\mathbf{x}_{\sigma}$-maximal vertices to describe the path. 

\begin{proof}[Proof of Theorem \ref{mainthm:startpt}]
Let $G_{n} \supseteq G_{n-1} \supseteq \dots \supseteq G_{0} = \mathbf{v}$ be the flag of faces from Corollary \ref{cor:flag}. Then, $\mathbf{v}$ is the $\mathbf{c}$-maximum of $G_{0}$. Define a path starting with $\mathbf{v}^{0} = \mathbf{v}$, and suppose that $\mathbf{v}^{i}$ is a $\mathbf{c}$-maximum of $G_{i}$, where $G_{i+1} \neq G_{i}$. Since $G_{i}$ is the $\mathbf{e}_{\sigma^{-1}(i)}$-maximal face of $G_{i+1}$, there is a $\mathbf{c}$-coherent $-\mathbf{e}_{\sigma^{-1}(i)}$-monotone path from $\mathbf{v}^{i}$ to a $\mathbf{c}$-maximum of $G_{i+1}$. Furthermore, by Corollary \ref{cor:lengthbound}, that path is of length at most $|\{\mathbf{e}_{\sigma^{-1}(i)}^{\T} \mathbf{x}: \mathbf{x} \in V(P)\}| -1 \leq |[k] \cup \{0\}| -1 = k$, since $P$ is a $(0,k)$-lattice polytope. Following paths of this form starting at $\mathbf{v}$ yields a path of length at most $k$ in at most $d$ distinct faces. Hence, the length of the path is at most $dk$.
\end{proof}

\subsection{Implementation} \label{subsec:imp}

For non-degenerate $(0,k)$-lattice polytopes, the description from what we have shown thus far suffices to develop an implementation by running the Simplex method twice with prescribed pivot rules. Unfortunately, degeneracy remains as an issue. Our paths are constructive, but finding such an edge in degenerate cases may be computationally expensive. In Section $3.4$ of \cite{01simplex}, by building on work in \cite{KleeKlein} and \cite{Murty2009}, we introduced a method for implementing shadow pivot rules such that the non-degenerate pivots will follow a $\mathbf{c}$-coherent $\mathbf{d}$-monotone path so long as one starts at the $\mathbf{c}$-maximum of the $\mathbf{d}$-minimal face of the polytope for any pair of linear objective functions $\mathbf{c}$ and $\mathbf{d}$. This method avoids cycling via attaching a lexicographic rule, but it does not necessarily avoid stalling (i.e., taking exponentially many degenerate pivots). Our approach to dealing with degeneracy is to split the LP into a sequence of LPs on faces of $P$ such that the non-degenerate steps across each LP follows exactly the path from Theorem \ref{mainthm:0ksimpmeth}.

\begin{proof}[Proof of Theorem \ref{mainthm:0ksimpmeth}]
To do this, first apply Phase $0$ of the Simplex method to find some $\mathbf{v} \in P$. Let $G_{n} = P$. Then we will follow a path of length $d \lfloor k/2 \rfloor$ to reach an $\mathbf{x}_{\sigma}$-maximal vertex for some $\sigma \in B_{n}$. First choose the coordinate $j_{1}$ such that $\mathbf{v}_{j_{1}}$ is closest to $0$ or $k$. If it is closest to $k$, apply the Simplex method with Bland's pivot rule to the LP $\max(\mathbf{e}_{j_{1}}^{\T} \mathbf{x}: \mathbf{x} \in G_{n})$ to reach the $\mathbf{e}_{j_{1}}$-maximal face $G_{n-1}$. If it is closest to $0$, apply the Simplex method with Bland's pivot rule to the LP $\max(-\mathbf{e}_{j_{1}}^{\T} \mathbf{x} : \mathbf{x} \in G_{n})$ to reach the $-\mathbf{e}_{j_{1}}$-maximal face $G_{n-1}$. Continue this until reaching $G_{i-1}$ by warm starting the LP on $G_{i}$ at the vertex found from maximizing on $G_{i+1}$, while choosing the coordinate closest to $0$ or $k$ that has not already been chosen. Such a path is of the type described in Lemma \ref{lem:GIpath} for reaching a nearest $\mathbf{x}_{\sigma}$-maximal vertex for some $\sigma \in B_{n}$. We may reconstruct $\sigma$ from the choice of $G_{i}$, and the length of the path is at most $d(\lfloor k/2 \rfloor)$.  

Once reaching $G_{0}$, the vertex in the path is exactly the $\mathbf{x}_{\sigma}$-maximal vertex $\mathbf{v}^{\sigma}$. Then to follow the path described in Theorem \ref{mainthm:startpt}, initialize $G_{1}$ at $\mathbf{v}^{\sigma}$ and apply the corresponding shadow pivot rule to follow the $\mathbf{c}$-coherent $-\mathbf{e}_{\sigma^{-1}(1)}$-monotone path from $\mathbf{v}^{\sigma}$ to the $\mathbf{c}$-maximum of $G_{1}$. Then to more generally go from a previously found $\mathbf{c}$-maximum of $G_{i-1}$ to the $\mathbf{c}$-maximum of $G_{i}$, initialize $G_{i}$ at that $\mathbf{c}$-maximum of $G_{i-1}$ and apply the corresponding to shadow-pivot rule to follow the $\mathbf{c}$-coherent $-\mathbf{e}_{\sigma^{-1}(i)}$-monotone path to a $\mathbf{c}$-maximum of $G_{i}$. By the results of Section $3.4$ of \cite{01simplex}, no cycling can occur on these paths, and the non-degenerate steps are exactly the corresponding coherent paths. Then the total length followed by the path is at most $dk$ by Theorem \ref{mainthm:startpt}. 
\end{proof}

This proposition shows that the existence of such a short path yields an algorithm for solving LPs on lattice polytopes even with degeneracy. A similar strategy also works to make the monotone diameter bounds for half-integral and $(m+1)$-level polytopes into algorithms. However, one would hope that there exists a single pivot rule for the Simplex method such that the number of non-degenerate steps in each of these cases satisfies the same or better bounds. We leave the existence of such a pivot rule as an open question. 

\section*{Acknowledgments}  
I would like to thank Jes\'{u}s De Loera, Sean Kafer, Laura Sanit\`{a}, and Stefan Weltge for useful discussions. I am also grateful for the excellent collaborative learning environment of the 2021 Tropical Geometry and Geometry of Linear Programming workshop at the Hausdorff Institute of Mathematics at Bonn and financial support from the NSF GRFP and NSF DMS-1818969. 

\bibliographystyle{amsplain}
\bibliography{References.bib} 


\end{document}